\documentclass[a4paper,12pt]{amsart}
\usepackage[utf8]{inputenc}
\usepackage{hyperref}
\usepackage{fullpage}
\usepackage{amsmath, amssymb}
\usepackage{mathtools} 
\usepackage[alphabetic,nobysame, initials]{amsrefs}
\usepackage{color,graphicx}

\newtheorem{thm}{Theorem}[section]
\newtheorem{lem}[thm]{Lemma}

\newtheorem{con}[thm]{Conjecture}
\theoremstyle{definition}
\newtheorem{rem}[thm]{Remark}
\newtheorem{defn}[thm]{Definition}

\renewcommand{\Re}{\mathbb R}
\renewcommand{\epsilon}{\varepsilon}
\newcommand{\Red}{\Re^d}
\newcommand{\Se}{\mathbb S}

\newcommand{\Sek}{\Se^{2}}
\newcommand{\Pe}{\mathbb P}
\newcommand{\Ze}{\mathbb Z}
\newcommand{\B}{\mathbf B}
\newcommand{\FF}{\mathcal F}
\newcommand{\GG}{\mathcal G}

\newcommand{\st}{\; : \; }
\renewcommand{\phi}{\varphi}
\newcommand{\vol}[1]{\operatorname{vol}\left(#1\right)}

\newcommand{\card}[1]{\left|#1\right|}

\newcommand{\inter}{\operatorname{int}}

\newcommand{\diam}{\operatorname{diam}}

\newcommand\abs[1]{\left|#1\right|}
\newcommand{\Hi}{\mathcal{H}}

\newcommand{\noshow}[1]{}
\newcommand{\maxmul}{C_1d\ln\left(\frac{1}{\varepsilon }\right)}

\title{Coverings: variations on a result of Rogers and 
on the Epsilon-net theorem of Haussler and Welzl}

\author[N. Frankl, J. Nagy, M. Naszódi]{N\'ora Frankl, J\'anos Nagy, M\'arton 
Nasz\'odi}
\address{
Dept. of Geometry,
Lorand E\"otv\"os University,
P\'azm\'any P\'eter S\'et\'any 1/C
Budapest, Hungary 1117
}
\email[N. Frankl]{aronlknarf@gmail.com}
\email[J. Nagy]{janomo4@gmail.com}
\email[M. Naszódi]{marton.naszodi@math.elte.hu}

\date{\today}

\begin{document}
\begin{abstract}
We consider four problems. Rogers proved that for any convex body $K$, we can 
cover ${\mathbb R}^d$ by translates of $K$ of density very roughly $d\ln d$. 
First, we 
extend this result by showing that, if we are given a family of positive 
homothets of $K$ of infinite total volume, then we can find appropriate 
translation vectors for each given homothet to cover ${\mathbb R}^d$ with the 
same (or, 
in certain cases, smaller) density.

Second, we extend Rogers' result to multiple coverings of space by translates 
of a convex body: we give a non-trivial upper bound on the density of the most 
economical covering where each point is covered by at least a certain number of 
translates.

Third, we show that for any sufficiently large $n$, the sphere ${\mathbb S}^2$ 
can be 
covered by $n$ strips of width $20n/\ln n$, where no point is covered too many 
times.

Finally, we give another proof of the previous result based on a combinatorial 
observation: an extension of the Epsilon-net Theorem of Haussler and Welzl. We 
show that for a hypergraph of bounded Vapnik--Chervonenkis 
dimension, in which each edge is of a certain measure, there is a not-too large 
transversal set which does not intersect any edge too many times.
\end{abstract}
\maketitle

\section{Introduction}

For a convex body $K$ we denote its translative covering density (the minimum 
density of the covering of $\Red$ by translates of $K$) by $\vartheta(K)$. We 
recall Rogers' estimate \cite{Ro57}: 
\begin{equation}\label{eq:rogers}
\vartheta(K)\leq d\ln{d}+d\ln\ln d+5d.
\end{equation}

Our first result is an extension of \eqref{eq:rogers}. For a family $\FF$ of 
sets in $\Red$, we say that $\FF$ \emph{permits a translative covering} of a 
subset $A$ of $\Red$ with density $\vartheta$, if we can select a translation 
vector $x_F\in\Red$ for each member $F$ of $\FF$ such that
$A\subseteq\bigcup\limits_{F\in\FF} x_F+F$, and the density of this covering is 
$\vartheta$.

\begin{thm}\label{thm:rogersallcases}
Let $K$ be a convex body in $\Red$, and let 
$\FF=\{\lambda_1 K, \lambda_2K, \dots \}$ ($0<\lambda_i$) be a 
family of its homothets with 
\[\sum_{i=1}^{\infty}{\lambda_i ^d}=\infty.\]
Let $\Lambda:=\{\lambda_1,\lambda_2,\ldots\}$.
\begin{enumerate}
\renewcommand{\theenumi}{\alph{enumi}}
 \item\label{item:rogerstranslates} 
If $\Lambda$ is bounded, and has a limit point 
other than zero, then $\FF$ permits a covering of space of density 
$\vartheta(K)$. 
 \item\label{item:rogersminis}
If $\Lambda$ is bounded, and has no limit point other than zero, then $\FF$ 
permits a covering of space of density one.
 \item\label{item:rogersmaxis}
If $\Lambda$ is unbounded, then $\FF$ permits a 
covering of space with maximum multiplicity $4d$ (that is, where no point is 
covered by more than $4d$ sets).
\end{enumerate}
\end{thm}
In case (\ref{item:rogersmaxis}), we will prove maximum multiplicity $2d$ in a 
special case which includes all smooth bodies, see Theorem~\ref{thm:bigcover}.
The proofs are in Section~\ref{sec:rogers}.

In the proof of Theorem~\ref{thm:rogersallcases}, we will use a result on 
covering $K$ by homothets of $K$.

\begin{thm}\label{thm:KcovKcor}
Let $K\subseteq \mathbb{R}^d$ be a convex body of volume one, and let $\FF$ be 
a 
family of 
positive homothets of $K$ with total volume at least 
\[\left\{ \begin{array}{ll} (d^3 \cdot \ln{d} \cdot  \vartheta(K)+e) 
2^d, & \textrm{ if } K=-K,\\
 d^3\cdot  \ln{d}\cdot \vartheta(K) \cdot \binom{2d}{d}+e\cdot 4^d, & \textrm{ 
in 
general.}
\end{array} \right. \]
Then $\FF$ permits a translative covering of $K$.
\end{thm}

This result is a strengthening of a result of \cite{Na10}, which, in turn is a 
strengthening of a result of Januszewski \cite{Ja03}. We prove it in 
subsection~\ref{subsec:KcovK}. We learned that a stronger bound was recently 
obtained by Livshyts and Tikhomirov \cite{LivTik16}.

Our second topic is multiple coverings of space. We denote the infimum of the 
densities of $k$-fold coverings of $\Red$ by translates of $K$ by 
$\vartheta^{(k)}(K)$. Apart from the estimate that follows from 
\eqref{eq:rogers} using the obvious fact $\vartheta^{(k)}(K)\leq 
k\vartheta(K)$, no general estimate has been known. 
For the Euclidean ball $\B_2^d$ in $\Red$, G. Fejes T\'oth 
\cites{FTG76,FTG79} gave the non-trivial lower bound 
$\vartheta^{(k)}(\B_2^d)>c_d k$ for some $c_d>1$, see more in the survey 
\cite{FTGhandbook}. We prove

\begin{thm}\label{thm:multiplecovspace}
Let $K\subseteq \Red$ be a convex body and $k \le d(\ln 
d+\ln \ln d)$. Then
\[\vartheta^{(k)}(K)\leq 6ed(3\ln d+\ln \ln d+15).\]
\end{thm}
This shows that G. Fejes T\'oth's bound (up to a constant factor) is sharp if 
$k=d\ln d$.

To prove Theorem~\ref{thm:multiplecovspace}, we present 
in Section~\ref{sec:multiple}
a more general statement, Theorem~\ref{thm:multiplecovgeneral}, which extends
\cite{AS}*{Theorem~1.6} and \cite{N15}*{Theorem~1.2}.

Our third topic is covering the sphere $\Sek:=\{x\in\Re^3\st |x|=1\}$ by 
strips. For a given point $x\in\Sek$, and $0\leq w\leq 1$, we call
$\{v\in\Sek\st \abs{\langle v,x \rangle}\leq w\}$ the \emph{strip} centered at 
$x$, of 
Euclidean half-width $w$.

\begin{thm}\label{thm:coverbystrips}
For any sufficiently large integer $N$, there is a covering of $\Sek$ by $N$ 
strips of Euclidean half-width $\frac{10\ln{N}}{N}$, with no point covered more 
than 
$c\ln{N}$ times, where $c$ is a universal constant.
\end{thm}

Our study of this question was motivated by a problem at the 2015 Mikl\'os 
Schweitzer competition posed by Andr\'as Bezdek, Ferenc Fodor, Viktor V\'igh 
and 
Tam\'as Zarn\'ocz on covering the two-dimensional sphere by strips of a given 
width such that no point is covered too many times.

We note the following dual version of Theorem~\ref{thm:coverbystrips}, and 
leave it to the reader to convince themselves that the two versions are 
equivalent:
\emph{
For any sufficiently large integer $N$, we can select $N$ points of $\Sek$ 
such that each strip of Euclidean half-width $\frac{10\ln{N}}{N}$ contains at 
least one and at most $c\ln{N}$ points, where $c$ is a universal constant.
}

In Section~\ref{sec:strips}, we present a direct, probabilistic proof of 
Theorem~\ref{thm:coverbystrips}. 

Our third topic, presented in Section~\ref{sec:vc}, is studying variants of the 
Epsilon-net theorem of Haussler and Welzl \cite{HaWe87}.

A set $X$ with a family $\Hi\subseteq 2^X$ of some of its subsets is called a 
\emph{hypergraph}, its Vapnik--Chervonenkis dimesion (VC-dimension, for short) 
is defined in Section~\ref{sec:vc}.  

\begin{thm}\label{thm:vcfewandmanyinedge}
Let $X$ be a set,  $\Hi\subset 2^X$ a hypergraph 
on $X$ of VC-dimension at most $d\geq2$, and $0<\varepsilon<1$. Let $\mu$ be a 
probability measure 
on $X$ with $\mu(H)=\varepsilon$ for each $H\in \Hi$.

a) If $\varepsilon \leq \frac{1}{d}$, then one can choose $\left\lfloor 
C\frac{d}{\varepsilon}\ln(1/\varepsilon)\right\rfloor$ elements of $X$ (not 
necessarily distinct), such that 
each edge of $\Hi$ contains at least one and at most $\maxmul$ chosen 
points (with multiplicity), where $C$ and $C_ 1$ are universal constants.

b) One can choose $\left\lfloor 
C\frac{d}{\varepsilon}\ln\left(\frac{1}{\varepsilon}+1\right)\right\rfloor$ 
elements of $X$ (not
necessarily distinct), such that 
each edge of $\Hi$ contains at least one and at most $C_1 d \ln 
d\ln\left(\frac{1}{\varepsilon}+1\right)$ chosen 
points (with multiplicity), where $C$ and $C_ 1$ are universal constants.
\end{thm}

The dual (and equivalent) version of Theorem~\ref{thm:coverbystrips} clearly 
follows from Theorem~\ref{thm:vcfewandmanyinedge}, since (using the uniform 
probability measure on $\Sek$) the measure of any strip of Euclidean half-width 
$w$ is $w$, and the VC-dimension of strips on $\Sek$ is bounded.

We also prove a similar result with essentially the same technique.

\begin{thm}\label{thm:vcfewinedge}
Let $X$ be a set,  $\Hi\subset 2^X$ a hypergraph 
on $X$ of VC-dimension at most $d\geq2$, and $N\geq 2$ an integer. Let $\mu$ be 
a probability measure on $X$ with $\mu(H)=1/N$ for each $H\in \Hi$.
Then one can find a multisubset of 
$N$ elements of $X$ (with 
multiplicity), such that 
each edge of $\Hi$ contains at most 
$Cd\frac{\ln N}{\ln{\ln{N}}}$ 
chosen points, where $C$ is a universal constant.
\end{thm}

It was pointed out to us by Nabil Mustafa that 
Theorems~\ref{thm:vcfewandmanyinedge} and \ref{thm:vcfewinedge} can be obtained 
directly from results on epsilon approximations.

\section{Covering space with given homothets -- Proof of 
Theorem~\ref{thm:rogersallcases}}\label{sec:rogers}

\begin{rem}\label{rem:maydropsets}
Given a family $\FF$ of compact sets in $\Red$. 
We want to cover the space with translates of members of $\FF$.
The minimum covering density that we can reach does not change whether 
we require that we use every member of $\FF$, or we may use only a subfamily.
Indeed, once we have a desired covering using a sub-family, we can take a 
zero-density arrangement of the rest of the members of $\FF$.
\end{rem}

In case (\ref{item:rogerstranslates}) of Theorem~\ref{thm:rogersallcases}, 
there is a subfamily of $\FF$ which consists of essentially translates of $K$. 
The proof of case (\ref{item:rogerstranslates}) now easily follows from 
Remark~\ref{rem:maydropsets}.


We make some preparations for the proof of case (\ref{item:rogersminis}).

\begin{defn} A collection $\mathcal{V}$ of Lebesgue-measurable subsets of 
$\Red$ is a \emph{regular family} if there is a constant $C$ for which 
$\diam(V)^d \le C \vol V$ holds for every $V\in \mathcal{V}$. 
\end{defn}

\begin{defn} A collection $\mathcal{V}$ of subsets of $\Red$ is a 
\emph{Vitali-covering} of $E\subseteq \Red$, if for every $x\in E$ and 
$\delta>0$, there is an element $U$ of $\mathcal{V}$ such that $x\in U$ and 
$0<\diam(U)<\delta$.
\end{defn}

We recall Vitali's covering theorem \cite{Vi08}.
\begin{thm}\label{thm:vitali}
Let $E\subset \Red$ be a measurable set with finite Lebesgue-measure, and 
let $\mathcal{V}$ be a regular family of closed subsets of $\Red$ that 
is a Vitali covering for $E$. Then there is a finite or countably infinite 
subcollection $\{U_j\}\subseteq \mathcal{V}$ of disjoint sets such that 
\[\vol{E \setminus \bigsqcup_j{U_j}}=0.\]
\end{thm}

\begin{proof}[Proof of (\ref{item:rogersminis}) of 
Theorem~\ref{thm:rogersallcases}]

We may assume that $\vol{K}=1$ and fix an $\varepsilon_0>0$.

For a subcollection $\GG$ of $\FF$
we denote by $\mathcal{G}_{\varepsilon}$ the subset of those elements 
of $\GG$ in which the ratio of homothety does not exceed $\varepsilon$.

Now, for every $\varepsilon>0$, the total volume in $\FF_{\varepsilon}$ is 
infinite. A bit more is true: for any subfamily $\GG$ of $\FF$ of infinite 
total volume, and for every $\varepsilon>0$, the total volume in 
$\GG_{\varepsilon}$ is infinite.

We partition 
\[\FF_{\varepsilon_0}=\big(\bigsqcup_i{\mathcal{A}_i}
\big)\bigsqcup\big(\bigsqcup_j{ \mathcal{
B}_j}\big)\] 
into countably many sub-collections, so that the total volume in each 
$\mathcal{A}_i$ and in each $\mathcal{B}_j$ is infinite.

We will cover most of the cube $[0,1]^d$ by a subfamily of $\mathcal{A}_1$, in 
which the sum of the volumes is at most $(1+\varepsilon_0\diam(K))^d$, and the 
rest of the cube by a subfamily of $\mathcal{B}_1$, in which the sum of the 
volumes is at most $\varepsilon_0$. If we can achieve this for any 
$\varepsilon_0>0$, the density bound for the whole space clearly follows.

Again, we partition 
\[\mathcal{A}_1=\bigsqcup_j{\mathcal{C}_j}\] into countably 
many subcollections, so that the total volume in each $\mathcal{C}_j$ is 
infinite.

Using Theorem~\ref{thm:KcovKcor}
for every $j\in \mathbb{N}$, we 
can cover the cube $[0,1]^d$ by the translates of the elements 
$(\mathcal{C}_{j})_{\frac{1}{j}}$. Since we use homothets of a fixed 
convex body, $K$, the union (over $j$) of these coverings is clearly 
a regular Vitali covering. 

Therefore we can apply Theorem~\ref{thm:vitali}.
There is a subcollection $\{U_l\}$ of disjoint sets for which
\[\vol{[0,1]^d\setminus \bigsqcup_l{U_l}}=\vol E=0.\]

Next, we will cover $E$ by a subcollection of $\mathcal{B}_1$, in which the sum 
of the volumes is at most $\varepsilon_0$. We partition 
$\mathcal{B}_1=\bigcup_l{\mathcal{D}_l}$ 
into countably many subcollection, each of infinite total volume. 

Since $\vol E=0$, for every $\varepsilon'>0$ there is collection 
$\mathcal{E}=\{K_1,K_2, \dots, \}$ of homothets of $K$ so that $E\subseteq 
\bigcup_l{K_l}$ and $\sum_{l}{\vol{K_l}}\le \varepsilon'$.

Note that $({\mathcal{D}_l})_{\varepsilon}$ is of infinite total volume for any 
$\varepsilon>0$.  
Thus, using Theorem~\ref{thm:KcovKcor}
for each $l$, we can cover $K_l$ 
by translates of members of a subfamily of $\mathcal{D}_l$ of total volume at 
most $C\vol{K_l}$ for some constant $C>0$. If $\varepsilon^{\prime}$ is small 
enough, we obtain a covering of $E$ of total volume $\varepsilon_0$, as 
promised.
\end{proof}

Case (\ref{item:rogersmaxis}) of Theorem~\ref{thm:rogersallcases} clearly 
follows from the following statement.

\begin{thm}\label{thm:bigcover}
Let $K$ be a convex body, and let 
$\FF=\{\lambda_1 K, \lambda_2 K, \dots \}$ be a family of its homothets so 
that the $\lambda_i$-s are not bounded. Then $\FF$ permits a 
translative covering of $\Red$ so that every point is covered at most $4d$ 
times.

Moreover, if $K$ is smooth at the points of 
intersection of $K$ with supporting hyperplanes that are parallel to one of the 
$d$ coordinate hyperplanes, then $\FF$ permits a translative covering of 
$\Red$ so that every point is covered at most $2d$ times.
\end{thm}
Clearly, if an affine image of $K$ has the special property that 
'coordinate-hyperplane touching points' are smooth, then the $2d$ bound on the 
covering multiplicity also follows.

\begin{proof}[Proof of Theorem~\ref{thm:bigcover} in the second case]

Fix $\varepsilon>0$. We may assume that $\FF$ has an element $\mu_0 K$ so 
that $Q_0=[-\varepsilon,\varepsilon]^d\subseteq \mu_0 K \subseteq [-1,1]^d$. We 
present an algorithm to produce the desired covering. We  will define 
inductively a sequence of cubes $Q_i$ ($i\in \mathbb{N}$), which are centered 
at 
the origin and have side length at least $i$, a sequence of translation vectors 
$x_1^1, x_1^2, \dots x_1^{2d}, x_2^1, \dots, x_2^{2d}, x_3^1, \dots$, and a 
sequence $\mu_1^1 K, \mu_1^2 K, \dots \mu_1^{2d} K, \mu_2^1 K, \dots, 
\mu_2^{2d} K, \mu_3^1 K, \dots$ of elements of $\FF$ so that the following hold
with the convention $x_0^j=0$ and $\mu_0^j=\mu_0$:
\begin{enumerate}
\item $Q_k\subseteq \bigcup_{i=0}^k\bigcup_{j=1}^{2d}{x_i^j+\mu_i^j K}$ for 
$k\in \mathbb{N}$ 
\item $\left(\bigcup_{i=0}^k\bigcup_{j=1}^{2d}{x_i^j+\mu_i^j K}\right) 
+B(0,\varepsilon) \subseteq Q_{k+1}$ for $k\in \mathbb{N}$
\item $\left(x_i^j +\mu_i^j K \right)\cap \left(x_i^{j+d} +\mu_i^{j+d} 
K\right)=\emptyset$ for $1\le j 
\le d$
\item $\left(x_i^j+\mu_i^j K\right) \cap \left(x_l^m+\mu_l^m 
K\right)=\emptyset$ if $|i-l|\ge 2$.
\end{enumerate}

\begin{figure}[t]
\includegraphics[width=.35\textwidth]{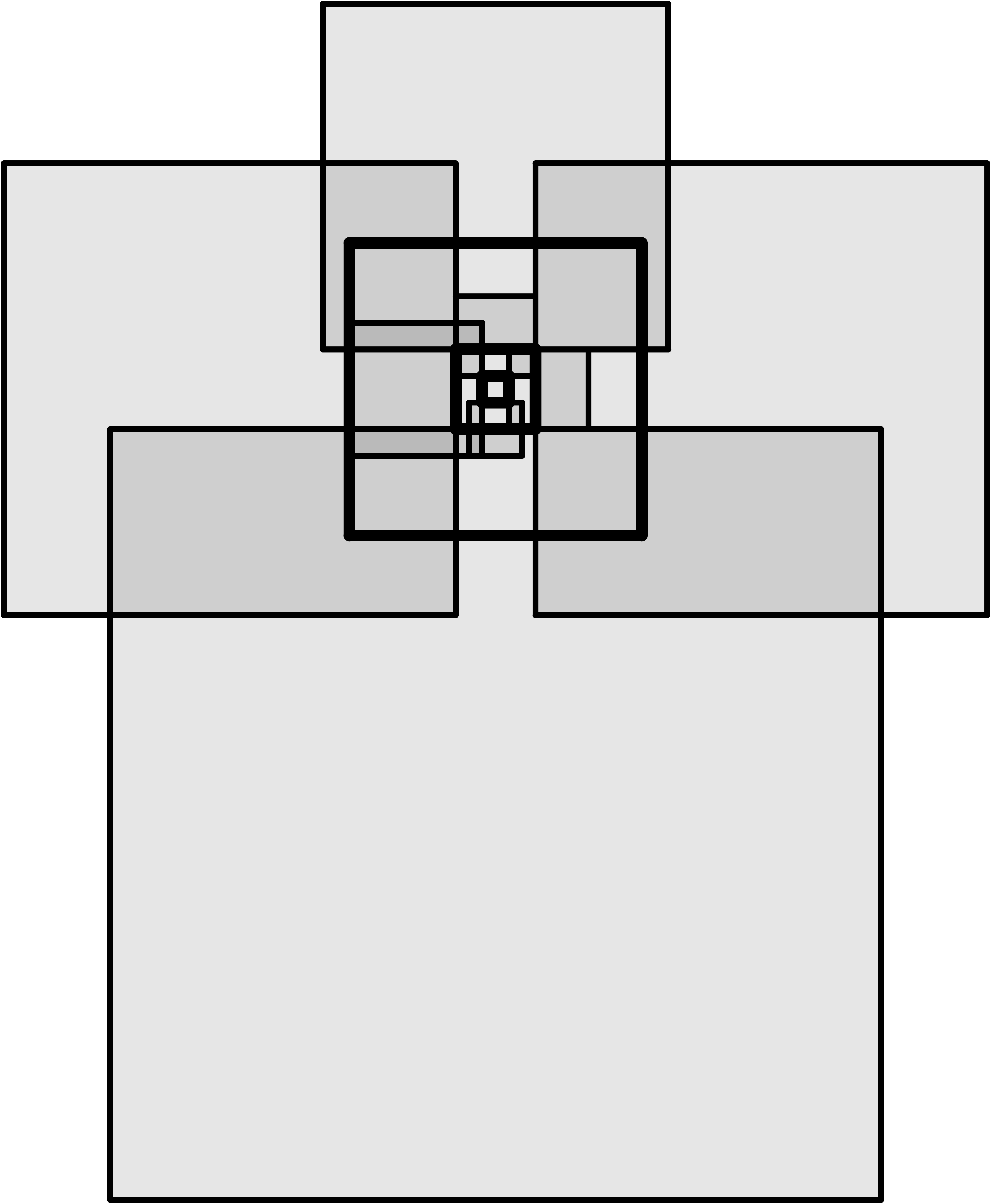}
\caption{The squares with bold edges are $Q_1,Q_2$ and $Q_3$ (counting from 
inside out).}\label{fig:squares}
\end{figure}

Indeed, assume that we found the $x_i^j$-s, $\mu_i ^j$-s and $Q_i$-s 
for $i\le k$. Choose $Q_{k+1}$ so that  
\[\left(\bigcup_{i=0}^k\bigcup_{j=1}^{2d}{x_i^j+\mu_i^j K}\right) 
+B(0,\varepsilon) 
\subseteq Q_{k+1}.\]

Let $H_i$ denote the support hyperplane of the $i$-th facet of $Q_{k}$, and 
$H_{i,+}$ the half-space bounded by $H_i$ that does not contain $Q_k$. Since 
the set of $\lambda_i$-s is unbounded, by the smoothness of $K$ at the touching 
points with the coordinate hyperplanes, we can choose 
a so far unused element $\mu_{k+1}^i K$ of 
$\FF$, and a translation 
vector $x_{k+1}^i$ such that 
\[Q_{k+1}\cap H_{i,+}\subseteq x_{k+1}^i+\mu_{k+1}^i K\]
and 
\[(x_{k+1}^i+\mu_{k+1}^i K) \cap (x_{k-1}^j+\mu_{k-1}^j K)=\emptyset\]
for all $j$.

Also we have that if $H_i$ $(i\le d)$ and $H_{i+d}$ support opposite sides of 
$Q_{k+1}$ then \[(x_{k+1}^i +\mu_{k+1}^i K) \cap (x_{k+1}^{i+d} 
+\mu_{k+1}^{i+d} 
K)=\emptyset,\] and \[Q_{k+1}\setminus Q_k \subseteq  \bigcup_{i=1}^{2d} 
{x_{k+1}^i +\mu_{k+1}^i K}.\]
Hence we can find the desired $Q_i$-s and translates.

Since $Q_i$ has side length at least $i$, 
\[\Red=\bigcup_{i=1}^{\infty}\bigcup_{j=1}^{2d}{x_i^j+\mu_i^j K}.\]

Property (3) ensures that, the subfamily $\bigcup_{i=1}^{2d}{x_{k}^i+\mu_k^i 
K}$ 
covers every point at most $d$ times, and property (4) yields that every 
point of $\mathbb{R}^n$ is covered by at most two subfamilies 
$\bigcup_{i=1}^{2d}{x_{k}^i+\mu_k^i K}$, which finishes the proof.
\end{proof}

\begin{rem} 
At first, one may believe that, by some approximation argument, 
the condition of smoothness can be dropped in Theorem~\ref{thm:bigcover}.
Unfortunately, this is not the case, the standard argument does not work. 

Let $K$ be a convex body in $\Red$ and $\FF=\{\lambda_1 K, 
\lambda_2 K, \dots\}$ a family of its homothets, such that the coefficients 
$\lambda_i$-s are not bounded. Let $L$ be a convex body with 
smooth boundary such that $L\subseteq K\subseteq (1+\varepsilon) L$. Consider 
the family $\FF^{\prime}=\{\lambda_1 L, \lambda_2 L, \dots\}$, and follow the 
steps of the proof of the smooth case in Theorem~\ref{thm:bigcover} for 
$\FF^{\prime}$.

We obtain a covering $x_i^j+\lambda_i^j L$ of $\Red$, where every point is 
covered at most $2d$ times. Then $x_i^j+(1+\varepsilon)\lambda_i^jL$ is also a 
covering. However, it may happen that $x_i^j+(1+\varepsilon)\lambda_i^jL$ 
covers every 
point infinitely many times: If $\lambda_i^k$ is sufficiently large, then 
$x_i^j+(1+\varepsilon)\lambda_i^j K$ may contain $B(0,i)$ for all $i$. 
\end{rem}

\begin{proof}[Proof of Theorem~\ref{thm:bigcover} in the general case]

We leave the proof of the following Lemma to the reader as an exercise. 
\begin{lem}\label{lem:smoothpoints}
Let $K\subseteq \Red$ be a convex body. Then there exist $j\le 2d$ points 
$\{x_1,x_2, \dots x_j\}$ on the boundary of $K$, so that $K$ is smooth in $x_i$ 
($1\le i \le j$) and $\bigcap_i {H_{i+}}=L$ is a bounded convex set with 
non-empty interior, where $H_{i+}$ is the half-space that contains $K$, bounded 
by the tangent hyperplane $H_i$ at $x_i$.
\end{lem}

Let $L$ be the polytope obtained in Lemma~\ref{lem:smoothpoints}. We may assume 
that $L$ contains the origin. Let $\varepsilon>0$ be fixed. We may also assume 
that $\FF$ has an element $\mu_0 K$ so that $-L_0=-\varepsilon L=\subseteq 
\mu_0 K \subseteq -L$. 

Similarly to the proof of the smooth case, we can inductively define a sequence 
$-\alpha_1 L, -\alpha_ 2L, \dots $ of 
homothets of $-L$, a sequence of translation vectors $x_1^1, x_1^2, \dots 
x_1^{2d}, x_2^1, \dots, x_2^{2d}, x_3^1, \dots$ and a sequence $\mu_1^1 K, 
\mu_1^2 K, \dots \mu_1^{2d} K, \mu_2^1 K, \dots, 
\mu_2^{2d} K, \mu_3^1 K, \dots$ of members of 
$\FF$, so that $\alpha_i\ge i$ and the following hold:

\begin{enumerate}
\item $-L_k\subseteq \bigcup_{i=0}^k\bigcup_{j=1}^{2d}{x_i^j+\mu_i^j K}$ for 
$k\in \mathbb{N}$ 
\item $\big(\bigcup_{i=0}^k\bigcup_{j=1}^{2d}{x_i^j+\mu_i^j K}\big) 
+B(0,\varepsilon) \subseteq -L_{k+1}$ for $k\in \mathbb{N}$
\item $\left(x_i^j+\mu_i^j K\right) \cap \left(x_l^m+\mu_l^m 
K\right)=\emptyset$ if $|i-l|\ge 2$.
\end{enumerate}

Now, the general case of Theorem~\ref{thm:bigcover} easily follows.
\end{proof}

\subsection{
\texorpdfstring{
Covering $K$ by its homothets}{Covering K by its homothets}}\label{subsec:KcovK}

\begin{thm}\label{thm:KcoversK}
For any $\varepsilon>0$, 
dimension $d$ and any convex body $K$ of volume one in $\Red$ with $o\in \inter 
K$, if a family $\FF$ of positive 
homothets of $K$ has total volume at least 
\[
d^2\left\lceil \frac{-\log{\varepsilon}}{\log{1+\varepsilon}} 
\right\rceil \vartheta(K)\frac{\vol{K-K}}{\vol K}
+\left(1+\frac{\varepsilon}{2}\right)^d 2^d\frac{\vol{K}}{\vol{K\cap K}}.
\]
then $\FF$ permits a translative covering of $K$.
\end{thm}

Theorem~\ref{thm:KcovKcor} clearly follows from this result. Indeed, we choose 
$\varepsilon=\frac{1}{d}$, and recall two facts. First, that there is a point 
$x\in K$ such that $\vol{K\cap (2x-K)}\ge \frac{1}{2^d}\vol K$. And second, 
that 
by \cite{RoS57}, $\vol{K-K}\leq \binom{2d}{d}\vol K$.

\begin{proof}[Proof of Theorem~\ref{thm:KcoversK}] 

First, we restate \cite{Na10}*{Theorem~1.3} in a slightly more general form 
than the original, which is easily obtained from the proof therein. The proof 
there easily yields this form. 
\begin{thm}\label{thm:nmsufficientcovering}
Let $K$ and $L$ be convex bodies in $\Red$ with $o \in \inter 
K$, and $\FF=\{\lambda_1 K, \lambda_2K,\ldots \}$ be a family of its 
homothets with $0<\lambda_i \leq\lambda_1<1$. Assume that 
\[\sum_{i=1}^{\infty}{\lambda_i^d}\ge 
2^d\frac{\vol{L+\lambda_1\frac{K\cap(-K)}{2}}}{\vol{K\cap (-K)}}.\] Then $\FF$ 
permits a translative covering of $L$.
\end{thm}

We fix $\varepsilon>0$. Now, we are given $\FF=\{\lambda_1 K, \lambda_2K,\ldots 
\}$ with 
$0<\lambda_i<1$ for all $i$. 
First, we consider the case when there is a subfamily 
$\FF^{\prime}=\{\mu_1 K, \mu_2 K, \dots \}$ of $\FF$ in which 
\[(1+\varepsilon)^{-1}\le \frac{\mu_i^d}{\mu_j^d}\le 
(1+\varepsilon)
\] 
for all $i$ and $j$, and 
\[\sum_{i=1}^{\infty}{\mu_i^d}\ge \vartheta(K) 
(1+\varepsilon)\frac{\vol{K-K}}{\vol K}.
\] 
In this case, $\FF^{\prime}$ has at least 
$\frac{1}{\mu_1^d(1+\varepsilon)}\vartheta(K) 
(1+\varepsilon)\frac{\vol{K-K}}{\vol K}=\frac{1}{\mu_1^d}
\vartheta(K)\frac{\vol{K-K}}{\vol K}$ members.

We may assume that $\mu_1$ is the smallest homothety ratio in $\FF^{\prime}$.
By the main result of \cite{RZ97}, we can cover $K$ by at most 
$\frac{\vol{K -\mu_1 K}\cdot \vartheta(K)}{\vol{\mu_1 K}}$ translates 
of $\mu_1 K$. The statement of the Theorem in this case easily follows.

Next, we assume that there is no such subfamily $\FF^{\prime}$. Consider the 
intervals 
\[(\varepsilon^d, \varepsilon^d(1+\varepsilon)], 
(\varepsilon^d(1+\varepsilon), \varepsilon^d(1+\varepsilon)^2], \dots, 
(\varepsilon^d(1+\varepsilon)^{c(\varepsilon)-1}, 
\varepsilon^d(1+\varepsilon)^{c(\varepsilon)}],
\] where 
$c(\varepsilon)=d\left\lceil \frac{-\log{\varepsilon}}{\log{1+\varepsilon}} 
\right\rceil$. 
Since 
$\varepsilon^d(1+\varepsilon)^{c(\varepsilon)}\ge 1$, we have 
\[\sum_{\lambda_i K\in \mathcal{F}, 
\lambda_i^d>\varepsilon^d}{\lambda_i^d}
<c(\varepsilon)d(1+\varepsilon)\vartheta(K)\frac{\vol{K-K}}{\vol K}.
\]
This implies that there exists a subfamily $\FF^{\prime}=\{\mu_1 K,\mu_2 K, 
\dots \}$, in 
which \[\mu_i^d \le \varepsilon^d\] and \[\sum_{i=1}^{\infty}{\mu_i^d}\ge 
\left(1+\frac{\varepsilon}{2}\right)^d 2^d\frac{\vol{K}}{\vol{K\cap(- K)}}.\] 
Then, by Theorem~\ref{thm:nmsufficientcovering}, $\FF^{\prime}$ permits a 
translative 
covering of $K$.
\end{proof}

\section{Multiple covering -- Proof of 
Theorem~\ref{thm:multiplecovspace}}\label{sec:multiple}

\begin{defn}
Let $\FF$ be a family of subsets of a base set $X$, and $k\in\Ze^+$. 
The \emph{$k$-fold covering number} of $\FF$, denoted by $\tau_k(\FF)$, is the 
minimum cardinality of a 
multi-subfmaily of $\FF$ such that each point of $X$ is contained in at least 
$k$ (with multiplicity) members of the subfamily.
\end{defn}
We recall that a \emph{fractional covering} of $X$ by $\FF$ is a mapping $w$ 
from $\FF$ to 
$\Re^+$ with $\sum_{x\in F\in\FF}w(F)\geq 1$ for all $x\in X$.
The total weight of a fractional covering is denoted by 
$w(\FF):=\sum_{F\in\FF}w(F)$, and its infimum is the 
\emph{fractional covering number} of $\FF$:
\[
\tau^\ast(X,\FF):=\inf\{w(\FF)\st w:\FF\rightarrow\Re^+ \mbox{ is a 
fractional-covering of } X\}. 
\]
For more on (fractional) coverings, cf. \cite{Fu88} in the abstract 
(combinatorial) setting and \cite{PaAg95} and \cite{Ma02} in the geometric 
setting.

We will use the following simple combinatorial statement.
\begin{lem}
Let $\FF$ be a family of subsets of a base set $X$ of fractional covering 
number $\tau^\ast:=\tau^\ast(\FF)$, 
and $k\in\Ze^+$. Then
\[
\tau_k\leq 
\left\lceil \tau^\ast \left( 
k+\frac{3}{2}\ln|X|+\frac{3}{2}\sqrt{(4k+\ln|X|)\ln|X|}\right)\right\rceil
\leq
\lceil 6\tau^\ast\max\{\ln |X|, k\}\rceil.
\]
\end{lem}
The proof is a standard probabilistic argument.
\begin{proof}
Let $w$ be a fractional covering of $X$ by $\FF$ of total weight 
$\tau^\ast:=\tau^\ast(\FF)$,
and let $m=\left\lceil \tau^\ast \left( 
k+\frac{3}{2}\ln|X|+\frac{3}{2}\sqrt{(4k+\ln|X|)\ln|X|}\right)\right\rceil$.

We pick $m$ members of $\FF$ randomly, independently with the same 
distribution: at each draw, each member $F$ of $\FF$ is
picked with probabilty $w(F)/w(\FF)$. For a fixed $x\in X$, the probability 
that $x$ is not covered at least $k$ 
times by the selected family is at most $\Pe(\xi<k)$, where 
$\xi=\xi_1+\ldots+\xi_m$, with independent 
random Bernouli (ie., 0/1--valued) variables $\xi_1,\ldots,\xi_m$, each of 
expectation $1/\tau^\ast$.
By Chernoff's inequality, $\Pe(\xi<k)\leq 
\exp\left(-\frac{(m-k\tau^\ast)^2}{3m\tau^\ast}\right)$. 
Thus, $\Pe(\mbox{there is an } x\in X \mbox{ which is not covered})\leq 
|X|\exp\left(-\frac{(m-k\tau^\ast)^2}{3m\tau^\ast}\right)$.
The lemma now clearly follows.
\end{proof}


For two sets $K$ and $L$ in $\Red$, we define $N_k(L,K)$, the \emph{$k$-fold 
covering number} of $K$ by $L$ as the minimum number of translates of $L$ that 
cover $K$ $k$-fold. Note that $N_k(L,K)=\tau_k(\FF)$, where $\FF=\{(x+K)\cap L 
\st x\in\Red\}$. We also define the \emph{fractional covering number} of $K$ by 
$L$ as $N^{\ast}(L,K)=\tau^{\ast}(\FF)$.

By \cite{AS}*{Theorem~1.7}, we have .
\begin{equation}\label{eq:easyfraccovering}
\max\left\{\frac{\vol{L}}{\vol{K}},1\right\}\le N^{\ast}(L,K)\le 
\frac{\vol{L-K}}{\vol{K}}
\end{equation}
for any Borel measurable sets, $K$ and $L$ is $\Red$.

The same proof as \cite{AS}*{Theorem~1.6} (or, \cite{N15}*{Theorem~1.2}) yields
\begin{thm}\label{thm:multiplecovgeneral}
Let $K$, $L$ and $T$ be bounded Borel measurable sets in 
$\Red$ and let $\Lambda\subset  \Red$ be a finite set with 
$L\subseteq \Lambda+T$. Then \[N_k(L,K)\le \lceil 6 N^{\ast}(L-T,K\thicksim 
T)\max\{\ln|\Lambda|,k\} \rceil.\]
If $\Lambda\subset K$, then we have \[N_k(L,K)\le \lceil 6 
N^{\ast}(L,K\thicksim T)\max\{\ln|\Lambda|,k\} \rceil.\]
\end{thm}

\begin{proof}[Proof of Theorem~\ref{thm:multiplecovspace}]
We may assume that 
\[
B(0,1)\subseteq K \subseteq \left [-d,d\right]^d.
\] 
Let $C=\left[-\frac{a}{2},\frac{a}{2}\right]^d$ be a cube of edge 
length $a$, where we will set $a$ later.

Let $\delta>0$ be fixed and let $\Lambda\subseteq \Red$ be a finite set 
such that $\lambda+\frac{\delta}{2}(K\cap(-K))$ is a saturated (ie. maximal) 
packing of 
$\frac{\delta}{2}(K\cap(-K))$ in $C-\frac{\delta}{2}(K\cap(-K))$. Thus 
$C\subseteq\Lambda+\delta K\subseteq \lambda+\delta(K\cap(-K))\subseteq 
\Lambda+\delta K$. 
By considering volume, we have that 
\[
|\Lambda|\le 
\frac{\vol{C-\frac{\delta}{2}(K\cap(-K))}}{\vol{\frac{\delta}{2}(K\cap(-K))}}
\le\frac{\left(a+\frac{\delta 
d}{2}\right)^d2^d}{\vol{B(0,1)}\left(\frac{\delta}{2}\right)^d}.\]

Equation \eqref{eq:easyfraccovering} yields that
\begin{equation}
\begin{multlined}
N^{\ast}(C-\delta(K\cap(-K)),K\thicksim \delta(K\cap(-K)))\le \\ 
N^{\ast}(C-\delta K, (1-\delta) K) \le 
\frac{\vol{C-K}}{\vol{(1-\delta) K}}\le 
\frac{\left(a+d\right)^d}{\left(1-\delta\right)^d\vol{K}}.
\end{multlined}
\end{equation}

From Theorem~\ref{thm:multiplecovgeneral} we have now 
\begin{equation}
\begin{multlined}
N_k(C,K)\le \left \lceil6 \frac{\left(a+d\right)^d}{(1-\delta)^d\vol{K}} \ln 
\left(\frac{\left(a+\frac{\delta d 
}{2}\right)^d2^d}{\vol{B(0,1)}\left(\frac{\delta}{2}\right)^d}
\right)\right\rceil \le \bigskip \\ 6 
\frac{\left(a+d\right)^d}{(1-\delta)^d\vol{K}} \ln 
\left(\frac{\left(a+\frac{\delta d 
}{2}\right)^d2^d}{\vol{B(0,1)}\left(\frac{\delta}{2}\right)^d}\right )+1.
\end{multlined}
\end{equation}
On the other hand 
\begin{equation}
\begin{multlined}
\vartheta^{(k)}(K)\le N_k(C,K)\frac{\vol{K}}{\vol{C}}\le 
6\frac{\left(a+d\right)^d}{\left(1-\delta\right)^d} 
\ln \left(\frac{\left(a+\frac{\delta d 
}{2}\right)^d2^d}{\vol{B(0,1)}\left(\frac{\delta}{2}\right)^d}\right)(\vol{C})^{
-1}+\frac{\vol{K}}{\vol{C}}.
\end{multlined}
\end{equation}

Choose now $\delta=\frac{1}{2d\ln d}$, $a=d^2$, and estimate $\vol{B(0,1)}$ by 
the volume of the cube of side length 
$\frac{1}{2\sqrt{d}}$, which is contained in $B(0,1)$.

\begin{equation}
\begin{multlined}
\vartheta^{(k)}\le 6 \left(\frac{d^2+d}{d^2}\right )^d \left( 1-\frac{1}{2d\ln 
d}\right)^{-1} \left(d \ln \left(4d^3\ln 
d+d+2+2d^{\frac{1}{2}}\right)\right)+1\le \medskip\\
6d \left( 1+\frac{1}{d}\right)^d\exp\left(\frac{1}{\ln d}\right)\ln \left(8d^3 
\ln 
d\right)\le \medskip\\ 
6d\left(1+\frac{2}{\ln d}\right)(3 \ln d+\ln \ln d+\ln 8) \le 
6ed(3\ln d+\ln \ln d+15)
\end{multlined}
\end{equation}
yields the desired bound.
\end{proof}

\section{Covering the sphere by strips -- A Direct Proof of 
Theorem~\ref{thm:coverbystrips}}\label{sec:strips}

In this section, we present a direct, probabilistic proof of 
Theorem~\ref{thm:coverbystrips}. We use the uniform probability measure on the 
sphere $\Sek$, and recall that the measure of any strip of Euclidean half-width 
$w$ is $w$.

Let $r=\frac{\ln{N}}{N}$. By a standard saturated packing argument, we may fix 
a set of points $v_1,\ldots,v_{N^2}$ on $\Sek$ such that the caps 
around $v_i$ of radius $r$ cover the sphere.

Let $X_i\;\; (i=1,\ldots,N)$ be independent random variables distributed 
uniformly on $\Sek$. We prove that with positive probability, the points $X_i$ 
will satisfy the conditions of the theorem.

First consider the following probability:\\
$Q_1 = \Pe(\exists v\in \Sek \st \abs{\langle v,X_j\rangle} \geq 
\frac{10\ln{N}}{N}, \mbox{ 
for all } j=1,\ldots, N)$.

Let\\ 
$P_i = \Pe(\exists v\in \Sek \st \abs{v-v_i} \leq r,   \abs{ \langle v, 
X_j\rangle } \geq 
\frac{10\ln{N}}{N},  \mbox{ for all } j=1,\ldots,N)$, where  $i=1,\ldots,N^2$.

The union of the events corresponding to $P_i$ covers the event corresponding 
to $Q_1$, 
because the caps around $v_i$ with radius $r$ cover the sphere. On the other 
hand, 
clearly, $P_i$ does 
not depend on $i$.
We obtain that  $N^2 P_1 \geq Q_1 $.

Assume that $\abs {v-v_i} \leq r $.

 \begin{equation*}
 \abs {\langle v, X_j\rangle} - \abs {\langle v-v_i, X_j\rangle}  \leq  \abs 
{\langle v_i, X_j\rangle},
\end{equation*}
thus,
 \begin{equation*}
 \abs {\langle v, X_j\rangle} - r \leq  \abs {\langle v_i, X_j\rangle}.
\end{equation*}
 
Hence, we can estimate from above $P_1$ as

\begin{equation*}
 P_1 \leq \Pe(\abs {\langle v_1, X_j\rangle} \geq 9r, j=1,\ldots,N)  = \left(1- 
9r\right)^{N} = \left(1- 9 
\frac{\ln{N}}{N}\right)^{N},
\end{equation*}

thus,

\begin{equation*}
 P_1  \leq e^{-9 \ln{n}} = N^{-9},
\end{equation*}
which yields $Q_1 \leq  \frac{1}{N^7} < \frac{1}{2}$.

Next, for any unit vector $v$, we denote the number of points in the strip $ 
\abs{ \langle v, X_j\rangle} \leq 10r$ from the set $X_1,\ldots,X_N$  by $k_v$, 
and 
denote the number of points in the strip $ \abs{\langle v, X_j\rangle} \leq 11r 
$ from the 
set $X_1,\ldots,X_N$  by $h_v$.
We  will bound from above the probability
$Q_2 = \Pe(\exists v\in \Sek | k_v \geq c \ln{N})$, where we will fix $c$ later.

Let $R_i = \Pe(\exists v\in \Sek | \abs {v-v_i} \leq r ,  k_v \geq c \ln{N} )$  
where  $i=1,\ldots,N^2 $.

Clearly, $R_i$ does not depend on $i$, and $N^2 R_1 \geq Q_2$.
So, we will estimate $R_1$ from above.

Assume that  $\abs {v-v_1} \leq r$.

\begin{equation*}
 \abs {\langle v, X_j\rangle} + \abs {\langle v-v_1, X_j\rangle}  \geq  \abs 
{\langle v_1, X_j\rangle},
\end{equation*}
thus,
 \begin{equation*}
 \abs {\langle v, X_j\rangle} + r \geq  \abs {\langle v_1, X_j\rangle}. 
\end{equation*}

We denote the floor of $c \ln{N}$ by $t$, and let 
$z= 11\ln{N}$.

\begin{equation*}
 R_1 \leq P(h_{v_i}  \geq c\ln{N})  \leq  \left(\frac{z}{N}\right)^t 
\binom{N}{t}
\end{equation*}

By Stirling's formula we easily get that there exists a universal constant $D$ 
such that $ \binom{N}{t}  \leq \frac{DN^N}{t^t (N-t)^{(N-t)}} $.
Thus,

\begin{equation*}
 R_1  \leq  \frac{D z^t \cdot N^N}{N^t \cdot t^t  (N-t)^{(N-t)}}  \leq  
\frac{D (e z)^t }{ t^t }.
\end{equation*}

If $ c \geq 100 $ then $ t \geq e^2 z $ , so we have $R_1  \leq   \frac{D 
}{ e^t }  \leq \frac{1}{N^3} $, if $N$ is large enough. It follows that

\begin{equation}
 Q_2 \leq N^2 R_1  \leq \frac{1}{2}.
\end{equation}

Overall, we obtained that $Q_1 +Q_2 <1 $, which means that if $N$ is 
large enough, then with positive probability the points $X_i$ will satisfy the 
the conditions of the theorem, with the constant $c=100$. This completes the 
proof of Theorem~\ref{thm:coverbystrips}.

Following the proof of  Theorem~\ref{thm:coverbystrips}, with a very little 
modification in the calculation, one can easily get the following theorem:

\begin{thm}\label{thm:football2}
There are $N$ points $x _i$  ($i=1,\ldots,N$)  on $\Sek$, such that for any 
unit 
vector $v$, there are at most $c\frac{\ln{N}}{\ln{\ln{N}}}$ chosen 
points in the strip $\abs{ \langle v, x\rangle } \leq \frac{1}{N}$, where $c$ 
is 
a 
universal constant.
\end{thm}

We pose the following open questions:

\begin{con}\label{con:alise}
There is a function $f$ on  the positive integers tending to infinity such 
that, for any $N$ points on $\Sek$, there is a 
unit vector $v$, for which the strip $\abs{\langle v,x\rangle} \leq 
\frac{1}{N}$ 
contains at 
least $f(N)$ of the given points.
\end{con}

\begin{con}\label{con:rebbeca}
There is a function $g$ on  the positive integers tending to infinity such that 
for any $N$ points in the unit disk on the plane,
there is a strip of width $\frac{1}{N}$ containing at least $g(N)$ of the given 
points. 
\end{con}

\begin{con}\label{con: Julia}
There is a function $h$ on  the positive integers tending to infinity such that 
for any $N$ points on $\Sek$ and any width $w>0$, there is a unit vector $v$, 
for which there are 
at least $h(N)$ given points in the strip $\abs{ \langle v, x\rangle} \leq w$,  
or there 
is no chosen point in the strip $\abs{\langle v, x> } \leq w$.
\end{con}

Note that  Conjecture~\ref{con:alise} would imply Conjecture~\ref{con: Julia} 
with $h=f$. This is because if $w \geq  \frac{1}{N}$ then the definition of $f$ 
guarantees 
that, and if $ w \leq \frac{1}{N}$, then computing the sum of the areas of the 
dual strips associated to  the points $x _i$ , $i=1,\ldots,N$, they cannot 
cover 
the sphere, so in that case there exists a unit vector $ v $ such that there is 
no chosen point in the strip $ \abs{\langle v, x\rangle } \leq w $.


\section{
\texorpdfstring{Some analogues of the Epsilon-net Theorem}{Some 
analogues of the Epsilon-net Theorem}
}\label{sec:vc}

In this section, we prove Theorems \ref{thm:vcfewandmanyinedge} and 
\ref{thm:vcfewinedge}. 
Both proofs closely follow the double-sampling technique of Haussler and Welzl 
from \cite{HaWe87}.

We recall some basics notions from the theory of hypergraphs, for details, 
we refer to \cite{Ma02}.

\begin{defn}\label{defn:vcdim}
The \emph{shatter function} of a hypergraph $\Hi$ on the set $X$ 
is $\pi_\Hi(m)=\max\limits_{A\subset X, |A|=m}|\Hi|_A|$.
The \emph{Vapnik--Chervonenkis dimension} (\emph{VC-dimension}, in short) of 
$\Hi$ is the maximal $m$ for which $\pi(m)=2^m$ (if there is no maximum, th 
VC-dimension is infinite).
\end{defn}

We recall the Sauer--Shelah Lemma \cites{Sau72,She72}.

\begin{lem}\label{lemma:Lisa}
Let $\Hi$ be a hypergraph of VC-dimension $d$. Then for any non-negative 
integer $m$,
$\pi(m)\le \binom{m}{0}+\binom{m}{1}+\ldots +\binom{m}{d}
\leq 2m^d$.
\end{lem}


In the proof of Theorem~\ref{thm:vcfewandmanyinedge} and 
Theorem~\ref{thm:vcfewinedge} we assume that the measure of every singleton is 
$0$, but this is not a restriction, because we can replace every singleton with 
measure greater than $0$ with an interval having the same measure.

\begin{proof}[Proof of Theorem~\ref{thm:vcfewandmanyinedge}]
We will assume, that the measure of every singleton of $X$ is $0$, to 
ensure that in a random sample the probability of having some element more than 
once is zero.
The general case will follow in the following way. If $A:=\{p_1,p_2,\ldots\}$ 
is the set of elements of $X$ that, as singletons, are of positive measure, 
then we replace each element, say $p_i$, of $A$ by a 'labeled' interval 
$p_i\times [0,1]$. This way, we obtain the set $\hat X$ from $X$. We define the 
measure of $\hat\mu$ on $\hat X$ in such a way that the measure $\mu(p_i)$ is 
uniformly distributed on the labeled interval $p_i\times [0,1]$. The set family 
$\hat\FF$ on $\hat X$ is essentially $\FF$, where if $F\in\FF$ contains $p_i$, 
then the corresponding $\hat F\in\hat\FF$ contains the entire labeled interval 
$p_i\times [0,1]$.

Next, with this assumption of having no positive-measure singletons, let 
$X_1,\ldots, X_{2N}$ be independent random variables according to $\mu$ taking 
values in $X$, where $N:=\left\lfloor 
C\frac{d}{\varepsilon}\ln(1/\varepsilon)\right\rfloor$. Set $Q_0:=\{X_1,\ldots, 
X_{N}\}, Q_1:=\{X_{N+1},\ldots, X_{2N}\}$ and $Q:=Q_0\cup Q_1$.

The Epsilon-net theorem yields that the probability that $Q_0$ is a 
\emph{transversal} to $\Hi$ (that is, that each edge $H\in\Hi$ 
intersects $Q_0$) is greater than $\frac{1}{2}$, if $C$ is sufficiently large.

For a given $H\in\Hi$, let $E_0^H$ be the event that $\card{Q_0\cap H} > 
\maxmul$, where $C_1>0$ is to be chosen later.
Let $E_1^H$ be the event that $E_0^H$ holds and 
$\card{Q_1\cap H} \leq 2\varepsilon N$.

Let $E_0$ be the union of the events $E_0^H$ for all 
$H\in\Hi$, that is, $E_0$ is the event that, for some $H\in \Hi$, we have 
$\card{Q_0\cap H} > \maxmul$.
Let $E_1$ be the union of the events $E_1^H$ for all $H\in\Hi$.

We claim that $\Pe(E_0)\le 2 \Pe(E_1)$.

Indeed, 
\begin{equation*}
 \frac{\Pe(E_1)}{\Pe(E_0)}=\Pe(E_1| E_0)\geq\min_{H\in\Hi}\Pe(E_1^H|E_0^H) >1/2,
\end{equation*}
by Markov's inequality.

Thus, it is sufficient to show that $\Pe(E_1)<\frac{1}{4}$ to obtain the 
theorem.

Next, we sample in a different way. We permute the indices of 
the variables $X_1 ,\ldots, X_{2N}$ with a random permutation (taking each 
permutation with equal probability), and denote the resulting variables as 
$Y_1,\ldots,Y_{2N}$.
They are again independent. We estimate the probability of the 
event $E_1$ for these variables.

We fix a $2N$-element subset $R$ of $X$, and let $L:=H\cap R$. We 
estimate the probability of the event $E_1^H$ under the condition that $Q=R$.
By Lemma~\ref{lemma:Lisa}, there are at most 
$2(2N)^d$ possibilities for $L$, so we have 
\begin{equation}\label{eq:pe1bound}
\Pe(E_1 | Q=R)\leq 2(2N)^d\max_{H\in\Hi} \Pe(E_1^H| Q=R).
\end{equation}

We fix $H\in \Hi$. If $t:=\card{L}<\maxmul$, then $E_1^H$ does not 
hold.
We consider the case when $t\geq \maxmul$.


In order to bound $\Pe(E_1^H| Q=R)$, we first note the following simple 
combinatorial fact. Let $V$ be a subset of $L$ with $0\leq m:=|V|\leq t$. 
Then,
\begin{equation*}
 \Pe(Q_1\cap H=V | Q=R)=\binom{2N-t}{N-m}/\binom{2N}{N}\leq
\binom{2N-t}{N-\lfloor t/2\rfloor}/\binom{2N}{N}
 \frac{D N^3}{2^{t}},
\end{equation*}
where $D$ is a universal constant. Since, for any $0\leq m\leq t$, we have 
$\binom{t}{m}$ ways to choose an $m$-element subset of $L$. Thus,

\begin{equation*}
 \Pe(E_1^H| Q=R)\le  
D \cdot \left(
\binom{t}{0}+\binom{t}{1}+\ldots+\binom{t}{\lfloor 2\varepsilon N\rfloor}
\right)2^{-t} N^3 \le 
3D\varepsilon N\binom{t}{\lfloor2\varepsilon N\rfloor}2^{-t} N^3,
\end{equation*}
using $t>4\varepsilon N$ if $C_1$ is large enough compared to $C$.
By Stirling's approximation, with the notations $l:=DN^3$, 
$k:=\lfloor2\varepsilon N\rfloor$ and $r:=t/k$, if $C_1$ is large enough 
compared to $C$, then $r > 10$, and the right hand side is less than

\begin{equation*}
\frac{3l \varepsilon N}{2^t} \cdot \frac{t^t}{k^{k} (t-k)^{t-k}} 
\leq
\frac{2lkr^k}{2^{rk}}\cdot\left(\frac{r}{r-1}\right)^{(r-1)k}
<
\left(\frac{3}{4}\right)^{(r-1)k} \cdot l<
\left(\frac{3}{4}\right)^{\frac{\maxmul}{2}} \cdot l,
\end{equation*}
if $C_1$ (and hence, $r$) is sufficiently large. 
So, by \eqref{eq:pe1bound}, it is sufficient to show, that if $C_1$ is large 
enough, then 
\begin{equation*}
2(2N)^d \left(\frac{3}{4}\right)^{\frac{\maxmul}{2}} \cdot DN^3< 1/4.
\end{equation*}
Clearly, it follows from
\begin{equation*}
D(4 C\frac{d}{\varepsilon}\ln(1/\varepsilon))^{d+3} 
\left(\frac{3}{4}\right)^{\frac{\maxmul}{2}} < 1/4.
\end{equation*}

The latter holds by the restriction $\varepsilon \leq 1/d$, if 
$C_1$ is large enough, because:
\begin{equation*}
\sqrt{D}(4 C d)^{d+3} \left(\frac{3}{4}\right)^{\frac{\maxmul}{4}} <
\frac{1}{4},
\end{equation*}
and
\begin{equation*}
\sqrt{D}\left(\frac{\ln(1/\varepsilon)}{\varepsilon}\right)^{d+3} 
\left(\frac{3}{4}\right)^{\frac{\maxmul}{4}} < 1.
\end{equation*}

Thus, by \eqref{eq:pe1bound}, we have that $\Pe(E_1 | Q=R)<1/4$. Since $R$ 
was arbitrary, we obtain $\Pe(E_1)<1/4$ completing the proof of the first part 
of Theorem~\ref{thm:vcfewandmanyinedge}.
The second part follows by the same calculation and the inequality
\begin{equation*}
D\left (4 
C\frac{d}{\varepsilon}\ln\left(\frac{1}{\varepsilon}+1\right)\right)^{d+3} 
\left(\frac{3}{4}\right)^{\frac{d\ln d\ln\left(\frac{1}{\varepsilon}+1\right) 
}{2}} < 1/4.
\end{equation*}
\end{proof}

\begin{proof}[Proof of Theorem~\ref{thm:vcfewinedge}]
Similarly to the proof of Theorem~\ref{thm:vcfewandmanyinedge}, we will 
assume, that the measure of every singleton of $X$ is $0$.

Let $X_1 ,\ldots, X_{N(\ln{N}+1)}$ be independent 
random variables according to $\mu$ taking values in $X$.  Set 
$Q_0:=\{X_1,\ldots, X_{N}\}, Q_1:=\{X_{N+1},\ldots, X_{N(\ln{N}+1)}\}$
and $Q:=Q_0\cup Q_1$. Denote by $r=Cd\frac{\ln{N}}{\ln{\ln{N}}}$, where $C$ is 
to be chosen later.



For a given $H\in\Hi$, let $E_0^H$ be the event that $\card{Q_0\cap H} > r$.
Let $E_1^H$ be the event that $E_0^H$ holds and 
$\card{Q_1\cap H}=\emptyset$.

Let $E_0$ be the union of the events $E_0^H$ for all 
$H\in\Hi$, that is, $E_0$ is the event that, for some $H\in \Hi$, we have 
$\card{Q_0\cap H} > r$.
Let $E_1$ be the union of the events $E_1^H$ for all $H\in\Hi$.

We claim that $\Pe(E_0)\le N^2 \Pe(E_1)$.

Indeed, 
\begin{equation*}
 \frac{\Pe(E_1)}{\Pe(E_0)}=\Pe(E_1| E_0)\geq\min_{H\in\Hi}\Pe(E_1^H|E_0^H)
>\left(\frac{N-1}{N}\right)^{N\ln{N}}\ge \frac{1}{N^2},
\end{equation*}
where, in the last inequality we used the fact that $(1-x/2)\ge e^{-x}$ for 
$0<x<1.59$.

Thus, it is sufficient to show that $\Pe(E_1)<\frac{1}{N^2}$ to obtain the 
theorem.

Next, we sample in a different way. We permute the indices of the variables 
$X_1 
,\ldots, X_{N(\ln{N}+1)}$ with a random permutation (taking each permutation 
with equal probability), and let the resulting variables be $Y_1,\ldots, 
Y_{N(\ln{N}+1)}$. They are again independent random variables. We will estimate 
the probability of the event $E_1$ for these variables.

We fix a subset $R$ of $X$ of $N(\ln{N}+1)$ elements, and let $L:=H\cap 
R$. We estimate the probability of the event $E_1^H$ under the condition that 
$Q=R$. By Lemma~\ref{lemma:Lisa}, there are at most $2(N\ln N)^d$ possibilities 
for $L=H\cap R$, so we have 
\begin{equation}\label{eq:pe1bound2}
\Pe(E_1 | Q=R)\leq 2(N\ln N)^d\max_{H\in\Hi} \Pe(E_1^H| Q=R).
\end{equation}

We fix $H\in \Hi$. If $t:=\card{L}\leq r$, then $E_1^H$ does not 
hold.

If  $t> r$, then 

\begin{equation*}
\Pe(E_1^H| Q=R)\leq
\frac{\binom{N}{t}}{\binom{N(\ln{N}+1)}{t}} \le 
\left(\frac{N}{N\ln{N}}\right)^{t} \le 
\left(\frac{1}{\ln{N}}\right)^{r} =
N^{-Cd}.
\end{equation*}

If $C$ is large enough, then by \eqref{eq:pe1bound2}, $\Pe(E_1 | Q=R)<1/N^2$. 
Since this bound does not depend on the choice of $R$, we have $\Pe(E_1)<1/N^2$ 
finishing the proof of Theorem~\ref{thm:vcfewinedge}.
\end{proof}

\section*{Acknowledgement}

M. Naszódi thanks the support of the Swiss National Science Foundation grants 
no. 200020-162884 and 200021-175977; the J\'anos Bolyai Research Scholarship of 
the Hungarian Academy of Sciences, and the National Research, Development, and 
Innovation Office, NKFIH Grant PD-104744. The three authors were partially 
supported by the National Research, Development, and Innovation Office, NKFIH 
Grant K119670. Part of this work is part of the MSc. thesis of N. Frankl.

We are grateful for the illuminating conversations that we had with Nabil 
Mustafa on epsilon nets, and with G\'abor Fejes T\'oth on geometric covering 
problems.

\bibliographystyle{amsalpha}
\bibliography{biblio}
\end{document}